\documentclass[11pt,reqno]{amsart}
\usepackage{graphicx,amsmath,amsthm,verbatim,amssymb,enumerate,mathtools}
\usepackage[french,english]{babel}
\usepackage{xcolor}
\definecolor{darkgreen}{rgb}{0,0.75,0}
\definecolor{darkred}{rgb}{0.75,0,0}
\definecolor{darkmagenta}{rgb}{0.5,0,0.5}
\makeatletter
\newenvironment{abstracts}{%
  \ifx\maketitle\relax
    \ClassWarning{\@classname}{Abstract should precede
      \protect\maketitle\space in AMS document classes; reported}%
  \fi
  \global\setbox\abstractbox=\vtop \bgroup
    \normalfont\Small
    \list{}{\labelwidth\z@
      \leftmargin3pc \rightmargin\leftmargin
      \listparindent\normalparindent \itemindent\z@
      \parsep\z@ \@plus\p@
      
      \itemsep\medskipamount
    }%
}{%
  \endlist\egroup
  \ifx\@setabstract\relax \@setabstracta \fi
}

\newcommand{\abstractin}[1]{%
  \otherlanguage{#1}%
  \item[\hskip\labelsep\scshape\abstractname.]%
}
\makeatother
\usepackage[colorlinks,citecolor=darkgreen,linkcolor=darkred,urlcolor=darkmagenta]{hyperref}
\numberwithin{equation}{section}

\newcommand{\mr}[1]{{\tt \href{http://www.ams.org/mathscinet-getitem?mr=#1}{MR#1}}}
\newcommand{\arxiv}[1]{{\tt \href{http://arxiv.org/abs/#1}{arXiv:#1}}}
\newcommand{\set}[1]{\left\{ #1 \right\}}

\newcommand{\Sett}[2]{\left\{ #1 \; : \; #2 \right\}}

\newcommand{\old}[1]{}

\newcommand{\abs}[1]{{\left\vert\kern-0.25ex #1
    \kern-0.25ex\right\vert}}
\newcommand{\df}{{d_f}} %fractal dimension	
\newcommand{\dw}{{d_w}} %walk dimension
\newcommand{\osc}{\operatorname{osc}}	
\DeclareRobustCommand{\SkipTocEntry}[5]{}
\newcommand{\G}{\mathcal{G}}
\newtheorem{theorem}{Theorem}[section]
\newtheorem{prop}[theorem]{Proposition}

\newtheorem{lemma}[theorem]{Lemma}
\newtheorem{lem}[theorem]{Lemma}
\newtheorem{corollary}[theorem]{Corollary}

\theoremstyle{remark}
\newtheorem*{remark}{Remark}
\newtheorem{example}{Example}

\numberwithin{counter}{section}

\theoremstyle{definition}
\newtheorem{definition}[theorem]{Definition}

\def\ind{\mathbf{1} } %indicator function
\def\00{\mathbf{0}}

\def\one{\mathbf{1}}

\def\E{\mathcal{E}}

\def\N{\mathbb{N}}

\def\R{\mathbb{R}}
\def\EE{\mathbb{E}}
\def\PP{\mathbb{P}}

\newcommand\norm[1]{\left\lVert#1\right\rVert} %norm
\begin{document}

\title[Anomalous heavy-tailed random walks]{Heat kernel estimates for anomalous heavy-tailed random walks}
\author[M. Murugan]{Mathav Murugan$^\ast$}
\thanks{$\ast$ M. Murugan was  partially supported by NSERC (Canada) 
 and the Pacific Institute for the Mathematical Sciences}\address{Department of Mathematics, University of British Columbia and Pacific Institute for the Mathematical Sciences, Vancouver, BC V6T 1Z2, Canada.}
\email[M. Murugan]{mathav@math.ubc.ca$^\ast$}

\author[L. Saloff-Coste]{Laurent Saloff-Coste$^\dagger$}
\thanks{$\dagger$ L. Saloff-Coste was partially supported by NSF grant DMS 1404435}\address{Department of Mathematics, Cornell University, Ithaca, NY 14853, USA.}
\email[L. Saloff-Coste]{lsc@math.cornell.edu}
\date{\today}
\date{\today}
\begin{abstracts}
\abstractin{english}
Sub-Gaussian estimates for the natural random walk is typical of many regular fractal graphs. Subordination shows that there exist heavy tailed jump processes whose jump indices are greater than or equal to two.
However, the existing machinery used to prove heat kernel bounds for such heavy tailed random walks fail in this case.
In this work we extend Davies' perturbation method to obtain transition probability bounds for these anomalous heavy tailed random walks.
We prove global upper and lower bounds on the transition probability density that are sharp up to constants.  An important feature of our work is that the methods we develop are robust to small perturbations of the symmetric jump kernel.
\abstractin{french}
Pour de nombreux graphes r\'{e}guliers de type fractal, la marche al\'{e}atoire simple satisfait de estimation de type sous-Gaussiennes. La technique de la subordination montre alors qu'il existe the processus t\^{e}te type stage dont l'indice des sauts est sup\'{e}rieur ou \'{e}gale a 2. Pour de telles processus, les techniques usuelles pour les estimations loin de la diagonale  ne fonctionnent pas. Nous \'{e}tendons la c\'{e}l\`{e}bre m\'{e}thode de Davies dans le cas de ces processus \`{a} sauts ``anormaux".  Nous obtenons des bornes sup\'{e}rieures et inférieures  pr\'{e}cises sur  le noyaux de transition pas des m\'{e}thodes qui sont stables sous de petites perturbations des sauts.   

\end{abstracts}

\maketitle

\section{Introduction}
The motivation behind this work is to obtain transition probability estimates for a class of random walks with heavy tailed jumps on fractal-like graphs. Many fractals satisfy the following sub-Gaussian estimates on the transition probability density $p_t$ for the natural diffusion: there exists $c_1,c_2,C_3,C_4>0$ such that
\begin{align}\label{e-subg}
\MoveEqLeft{\frac{c_1}{t^{\df/\dw}} \exp \left( - \left( \frac{d(x,y)^\dw}{c_2 t}\right)^{1/(\dw-1)} \right) } \nonumber \\
& \le p_t(x,y) \le \frac{C_3}{t^{\df/\dw}} \exp \left( - \left( \frac{d(x,y)^\dw}{C_4 t}\right)^{1/(\dw-1)} \right)
\end{align}
for all points $x,y$ in the underlying space and for all times $t>0$ \cite{Bar0,BP,BB1,HK,Kig}. The parameters $\df,\dw > 0$ are intrinsic and depend on the geometry of the underlying space.  Sub-Gaussian estimates similar to \eqref{e-subg} were obtained in the discrete time setting for the nearest neighbor walks on fractal like graphs \cite{Bar,Jon,GT,GT0,BCK,BB2}. A precise formulation of these sub-Gaussian estimates for graphs is provided in \ref{usg} and \ref{lsg} in Section \ref{ss-example}.

If $\mathcal{L}$ denotes the generator of a diffusion whose heat kernel satisfies the sub-Gaussian estimate \eqref{e-subg}, then  for any $\beta \in (0,\dw)$, the
operator $- (-\mathcal{L})^{\beta/\dw}$,
is the generator of a non-local Dirichlet form, and its
heat kernel admits the estimate
\begin{equation} \label{e-ht}
p_t(x,y) \asymp  \frac{1}{t^{\df/\beta}} \left( 1 + \frac{d(x,y)}{t^{1/\beta}} \right)^{- (\df+\beta)}
\end{equation}
for all times $t$ and for all points $x,y$ in the underlying space (see, for example, \cite{Sto}, \cite{Kum}, \cite[Lemma 5.4]{Gri}).
Here $\asymp$ means that the ratio between both sides of the expression is bounded above and below by finite positive constants that do not depend on $x,y,t$.
Estimates similar to \eqref{e-ht} can be established using subordination in the discrete time context as well \cite[Section 5]{MS2}.
Furthermore, it was shown in \cite{GK} that \eqref{e-subg} and \eqref{e-ht} exhaust all possible two-sided estimates of heat kernels of self-similar type.
We call the parameter $\beta$ in \eqref{e-ht}  the \emph{jump index}.

Obtaining estimates of the form \eqref{e-ht} for the case $\beta \in (0,2)$ has received much attention recently, both in the context of continuous time jump processes (see, for example, \cite{CK1,CK2,BBK,BGK})  and discrete time heavy tailed random walks (see \cite{BL,MS1}). For the case $\beta \in (0,2)$, Davies' method is used to obtain estimates of the form \eqref{e-ht} which are stable under bounded perturbations of the corresponding Dirichlet form. However it is known that the use of Davies' method to obtain the upper bound in \eqref{e-ht} is no longer adequate for the case $\beta \ge 2$ (see \cite[Section 1]{GHL},
\cite[Remark 1.4(d)]{MS2}).

 Recently characterizations of bounds of the form \eqref{e-ht} for the case $\beta \ge 2$  have been given in  \cite[Theorem 2.1 and Theorem 2.3]{GHL}.
However these characterizations rely on probabilistic estimates on the exit times (called survival estimates and tail estimates in \cite{GHL}).
 It is not known if these survival and tail estimates are stable under bounded perturbation of the jump kernel of the jump process in the continuous time case or the conductance of the heavy tailed random walk in the discrete time case.
 Therefore there is a need to develop techniques to obtain \eqref{e-ht} that are robust under  bounded perturbation of the Dirichlet form for the case $\beta  \ge 2$ -- see Example \ref{x-one} for a concrete example that compares our result with previous works. 

 The main goal of the work is to show that the estimate \eqref{e-ht} is stable under bounded perturbation of conductances for a heavy tailed random walk in  the case $\beta \ge 2$, under some natural hypothesis. This is carried out by modifying the Davies' perturbation method using a cutoff Sobolev inequality. We extend the techniques developed in \cite{MS3} to a non-local setting. Although our work concerns discrete time heavy tailed random walks, we expect that the techniques we develop extends to continuous time jump processes as well.
For instance, in the proof of upper bound  (see Proposition \ref{p-offd}), we obtain corresponding continuous time bounds -- cf. \eqref{e-od1}. Indeed, the modified Davies' method developed in this work can be implemented for continuous time processes in the metric measure space setting without any new essential difficulties \cite{HL}.

To state the results precisely, we recall some standard notions concerning graphs and Markov chains.
Let $G=(X,E)$ be an infinite, simple, connected, locally finite graph. The elements of the set $X$ are called vertices.  Some of the vertices are connected by an edge, in which case
we say that they are neighbors.
Let $d(x,y)$ be the graph distance between points $x,y \in X$, that is the minimal number of edges in any edge path connecting $x$ and $y$. 
 Consider a measure $\mu$ on $X$. We sometimes abuse notation and consider $\mu$ as a function on $X$ by setting $\mu(x)$ to be $\mu\left( \set{x}\right)$.
 Denote the metric balls and their measures as follows
\[
B(x,r):= \{ y \in X : d(x,y) \le r \}\hspace{5mm} \operatorname{and}\hspace{5mm} V(x,r):= \mu(B(x,r))
\]
for all $x \in X$ and $r \ge 0$. 
For convenience, we refer to the graph $(X,E)$ endowed with the graph distance $d$ and a measure $\mu$ as a metric measure space $(X,d,\mu)$ which we call a \emph{vertex weighted graph}.
We consider a vertex weighted graph $(X,d,\mu)$ that satisfies the polynomial volume growth assumption:
there exists $\df > 0$, $C_V>0$ such that
\begin{equation} \label{df}
C_V^{-1} r^{\df} \le V(x,r) \le C_V r^\df \tag*{$\operatorname{V}(\df)$}
\end{equation}
for all $r \ge 1$ and for all $x \in X$. We use the notation $\mathbb{N}= \set{0,1,2,3,\ldots}$ and $\mathbb{N}^*= \set{1,2,3,\ldots}$.
\begin{definition}[Symmetric Markov operator] \label{d-mark}
We say that $K:\ell^\infty(X) \to \ell^\infty(X)$ is a \emph{$\mu$-symmetric sub-Markov operator} on a vertex weighted graph $(X,d,\mu)$ if there exists a \emph{kernel} $k:X \times X \to [0,\infty)$ that satisfies
\begin{align*}
k(x,y)&=k(y,x) \hspace{3mm}\mbox{for all $x,y \in X$,} \\
\sum_{y \in X} k(x,y) \mu(y) &\le 1  \hspace{3mm}\mbox{for all $x \in X$,} \\
Kf(x) &= \sum_{y \in X} k(x,y) f(y) \mu(y)  \hspace{3mm}\mbox{for all $x \in X, f \in \ell^\infty(X)$.} 
\end{align*}
If the inequality above is replaced by an equality we say that $K$ is a \emph{$\mu$-symmetric Markov operator} on a vertex weighted graph.

Given a set $U \subset X$ and a $\mu$-symmetric Markov operator $K$ on $X$ with kernel $k$, we define the $\mu$-symmetric sub-Markov operator $K_U$ as the operator with kernel $k_U(x,y)=k(x,y) \one_U(x) \one_U(y)$. 
\end{definition}
There is a natural random walk $(Y_n)_{n \in \N}$ on $X$ associated with a $\mu$-symmetric Markov operator $K$ with kernel $k$. The Markov chain $(Y_n)_{n \in \N}$  is defined by the following
one-step transition probability
\[
 \PP_{x}(Y_1=y):= \PP\left(Y_1=y \hspace{1mm}\vline\hspace{1mm} Y_0=x \right) = k(x,y)\mu(y).
\]
The sub-Markov operator $K_U$ then corresponds the Markov chain $(Y_n)$ killed upon exiting $U$.

 We say that a  $\mu$-symmetric Markov operator $K$ on a vertex weighted graph $(X,d,\mu)$ satisfies \ref{j-beta}, if there exists a constant $C>0$ such that the corresponding kernel $k$ satisfies
\begin{equation} \label{j-beta}
C^{-1} \frac{1}{(1+d(x,y))^{\df+\beta}} \le k(x,y)= k(y,x) \le  C\frac{1}{(1+d(x,y))^{\df+\beta}} \tag*{$\operatorname{J}(\beta)$}
\end{equation}
for all $x,y \in X$, where $\df$ is the volume growth exponent given in \ref{df}.
If $K$ satisfies \ref{j-beta} for some $\beta>0$, then we say that $\beta$ is the \emph{jump index} of the random walk driven by $K$.

%Statement of main result:
Let $k_n(x,y)$ denote the kernel of the iterated power $K^n$ with respect to the measure $\mu$. We are interested in obtaining  estimates on $k_n(x,y)$ for all values of $n\in\N^*$ and for all $x,y \in X$.   We say that a $\mu$-symmetric Markov operator $K$ satisfies \ref{hkp}, if there is a constant $C>0$ such that the iterated kernel $k_n$ satisfies the estimate
\begin{equation} \label{hkp}
 C^{-1}\min\left( \frac{1}{n^{\df/\beta}},\frac{n}{(d(x,y))^{\df+\beta}}\right) \le k_n(x,y) \le  C\min\left( \frac{ 1}{n^{\df/\beta}},\frac{ n}{(d(x,y))^{\df+\beta}}\right) \tag*{$\operatorname{HKP}(\df,\beta)$}
\end{equation}
for all $x,y \in X$ and for all $n \in \N^*$, where $\df$ is the volume group exponent given by \ref{df}. We say that $K$ satisfies \hypertarget{uhkp}{$\operatorname{UHKP}(\df,\beta)$}, if  the iterated kernel $k_n$ satisfies the upper bound in \ref{hkp}. Similarly, we say that $K$ satisfies \hypertarget{lhkp}{$\operatorname{LHKP}(\df,\beta)$}, if the iterated kernel $k_n$ satisfies the lower bound in \ref{hkp}. Note that \ref{hkp} is same as the bounds described in \eqref{e-ht}, since
\begin{align*}
 \frac{1}{n^{\df/\beta}} \left( 1 + \frac{d(x,y)}{n^{1/\beta} } \right)^{-(\df+\beta)}  &\le  \min\left( \frac{1}{n^{\df/\beta}}, \frac{n}{(d(x,y))^{\df + \beta}}\right),\\
\min\left( \frac{1}{n^{\df/\beta}}, \frac{n}{(d(x,y))^{\df + \beta}}\right)& \le 2^{\df+\beta}  \frac{1}{n^{\df/\beta}} \left( 1 + \frac{d(x,y)}{n^{1/\beta} } \right)^{-(\df+\beta)}.
\end{align*}

The goal of this work is to develop methods to obtain the bound \ref{hkp} that are robust to small perturbations of Dirichlet form in the sense given by \ref{j-beta}.
Similar to the anomalous diffusion setting \cite{MS3}, we rely on a cutoff Sobolev inequality to implement Davies' method.
 To introduce cutoff Sobolev inequality, we first need to define cutoff functions and energy measure.
\begin{definition}[Cutoff function]
Let $U \subset V$ be finite sets in $X$. We say that $\phi: X \to \R$ is a cutoff function for $U \subset V$ if  $\phi \equiv 1$ on $U$ and $\phi \equiv 0$ on $V^\complement$.
\end{definition}
\begin{definition}[Energy measure]\label{d-ener}
For a $\mu$-symmetric Markov operator $K$ with kernel $k(x,y)$ with respect to $\mu$, and for functions $f,g \in L^\infty(X,\mu)$, we define the energy measure corresponding to $K$ as the function
\[
\Gamma(f,g)(x) = \frac{1}{2} \sum_{y \in X}(f(x)-f(y))(g(x)-g(y))  k(x,y)\mu(y)\mu(x)
\]
for all $x \in X$. The function $\Gamma(f,g)$ can be considered as a (signed) measure where the measure of the singleton $\set{x}$ is $\Gamma(f,g)(x)$. 
\end{definition}
We introduce a cutoff Sobolev inequality that plays an important role in obtaining \ref{hkp}. Cutoff Sobolev inequalities were first introduced by Barlow and Bass in  \cite{BB3} for graphs and then extended by Barlow, Bass and Kumagai \cite{BBK} to metric measure spaces. Recently Andres and Barlow simplified the cutoff Sobolev inequalities in \cite{AB}. We obtain transition probability estimates using the following cutoff Sobolev inequality.
The motivation behind studying cutoff Sobolev inequalities in \cite{BB3,BBK,AB} is that they provide a method to obtain sub-Gaussian estimates that is robust to bounded perturbation of the Dirichlet form.
We are motivated by similar reasons to formulate a version of cutoff Sobolev inequality for heavy tailed random walks.
Our definition below is inspired by the cutoff Sobolev annulus inequality in \cite{AB} and a self-improving property of cutoff Sobolev inequality \cite[Lemma 5.1]{AB}.
\begin{definition}[Cutoff Sobolev inequality] \label{d-csj} 
Let $K$ be a $\mu$-symmetric Markov operator  on a vertex weighted graph $(X,d,\mu)$ and let $\Gamma$  denote the corresponding energy measure. Let $\beta \ge 2$. We say that $K$ satisfies the cutoff Sobolev inequality \hypertarget{csj}{$\operatorname{CSJ}(\beta)$} if it satisfies the following property:
for any $x \in X$ and for any $R,r>0$, there exists a cutoff function $\phi$ for $B(x,R) \subset B(x,R+r)$ such that
\begin{align}
\phi & \equiv 1 \hspace{3mm} \mbox{ in } B(x,R+r/2),\label{e-csa1} \\
\phi & \equiv 0 \hspace{3mm} \mbox{ in } B(x, R+9r/10)^\complement, \label{e-csa2}
\end{align}
and for any function $f \in L^2(X,\mu)$
\begin{equation} \label{e-csa3}
\sum_{y \in U} f^2(y) \Gamma_U (\phi,\phi)(y)  \le C_1 \sum_{y \in U} \Gamma_U(f,f)(y) +  \frac{C_2}{r^\beta}  \sum_{y \in U} f^2(y) \mu(y),
\end{equation}
where $U= B(x,R+r) \setminus B(x,R)$ and $\Gamma_U$ is the energy measure corresponding to the sub-Markov operator $K_U$.
\end{definition}
The condition \hyperlink{csj}{$\operatorname{CSJ}(\beta)$} is useful because it preserved under bounded perturbation of the Markov kernel.
The appearance of the expression $r^\beta$ in the second term of \eqref{e-csa3} signifies the role played by $\beta$ in the  space-time scaling relation of the walk. 

The main result is the following characterization of \hyperlink{uhkp}{$\operatorname{UHKP}(\df,\beta)$} using \hyperlink{csj}{$\operatorname{CSJ}(\beta)$}.
\begin{theorem} \label{t-main3}
	Let $(X,d,\mu)$ be a vertex weighted graph satisfying \ref{df} with volume growth exponent $\df$.  Let $K$ be $\mu$-symmetric Markov operators whose kernel with respect to $\mu$ satisfies \ref{j-beta} for some $\beta \ge 2$. Then $K$ satisfies  \hyperlink{uhkp}{$\operatorname{UHKP}(\df,\beta)$} if and only if $K$ satisfies \hyperlink{csj}{$\operatorname{CSJ}(\beta)$}.
\end{theorem}
As a corollary we show that \ref{hkp} is stable under bounded perturbation of the Markov kernel.
\begin{corollary}[Stability of  \ref{hkp}] \label{c-main2}
Let $(X,d,\mu)$ be a vertex weighted graph satisfying \ref{df} with volume growth exponent $\df$.
Let $K_1$ and $K_2$ be two $\mu$-symmetric Markov operators whose kernels with respect to $\mu$ satisfy \ref{j-beta} for some $\beta \ge 2$.
 Then $K_1$ satisfies \ref{hkp} if and only if $K_2$ satisfies \ref{hkp}.
\end{corollary}
In Section \ref{ss-example}, we show how Theorem \ref{t-main3} and Corollary \ref{c-main2} apply to a large family of graphs that satisfy sub-Gaussian estimates for the simple random walk.

\section{Cutoff Sobolev inequalities for heavy-tailed random walks}
\label{s-csj}
In this section, we prove the implication \hyperlink{uhkp}{$\operatorname{UHKP}(\df,\beta)$} $\implies$ \hyperlink{csj}{$\operatorname{CSJ}(\beta)$} in Theorem \ref{t-main3}.
Then we show a self-improving property of \hyperlink{csj}{$\operatorname{CSJ}(\beta)$}. 

\begin{prop} \label{p-csa}
	Let $(X,d,\mu)$ be a vertex weighted graph satisfying \ref{df}. Let $K$ be a $\mu$-symmetric Markov operator satisfying \hyperlink{uhkp}{$\operatorname{UHKP}(\df,\beta)$}, for some $\beta \ge 2$. Then $K$ satisfies \hyperlink{csj}{$\operatorname{CSJ}(\beta)$}.
\end{prop}

We follow the approach of Andres and Barlow in \cite[Section 5]{AB} to prove the cutoff Sobolev inequality in Proposition \ref{p-csa}.
The first step is to obtain estimates on exit times.
We define the exit time for the Markov chain $(Y_n)_{n \in \N}$ as
\[\tau^Y_{B(x,r)} = \min \Sett{k \in \mathbb{N}} { Y_k \notin B(x,r)}.
\]
The exit time $\tau^Y_{B(x,r)}$ satisfies the following survival estimate.
\begin{lem} \label{l-sur}
Let $(X,d,\mu)$ be a vertex weighted graph satisfying \ref{df} with volume growth exponent $\df$.
Let $K$ be $\mu$-symmetric Markov operator satisfying \hyperlink{uhkp}{$\operatorname{UHKP}(\df,\beta)$},  for some $\beta \ge 2$.
Let $(Y_n)_{n \in \mathbb{N}}$ denote the Markov chain driven by the operator $K$.
There exist constants $\epsilon, \delta \in (0,1)$ such that for all $x \in X$ and for all $r \ge 1$,
\begin{equation}\label{e-sur}
\PP_x \left(  \tau^Y_{B(x,r)} \le  \delta r^\beta \right) \le \epsilon.
\end{equation}
\end{lem}
\begin{proof}
We follow the argument  in \cite[pp. 15]{BGK}.
By \hyperlink{uhkp}{$\operatorname{UHKP}(\df,\beta)$}  and  \ref{df}, there exists $C_1, C_2 \ge 1$ such that
\begin{equation} \label{e-sur1}
\PP_x (d(Y_n,x) \ge r) = \sum_{y \notin B(x,r)} k_n(x,y) \mu(y) \le C_1 \sum_{y \in B(x,r)^\complement} \frac{n \mu(y)}{ d(x,y)^{\df+ \beta}} \le C_2 \frac{n}{r^\beta}
\end{equation}
for all $n \in \N^*$ and for all $x \in X$. By \eqref{e-sur1} and the strong Markov property of $\set{Y_k}$ at time $\tau=\tau^Y_{B(x,r)}$, there exists $C_3 \ge 1$ such that
\begin{align*}
\PP_x(\tau \le n) &\le \PP_x \left( \tau \le n , d(Y_{2n},x) \le r/2 \right) + \PP_x \left( d(Y_{2n},x) > r/2\right) \\
&\le \PP_x\left( \tau \le n, d(Y_{2n},Y_\tau) \ge r/2 \right) + 2^{1+\beta} C_2 n/r^\beta \\
&= \mathbb{E}^x \left( \ind_{\set{\tau \le n}} \PP_{Y_\tau} \left( d (Y_{2n-\tau},Y_0) \ge r/2 \right) \right) + 2^{1+\beta} C_2 n  /r^\beta \\
& \le \sup_{y \in B(x,r)^\complement} \sup_{s \le n} \PP_y\left( d (Y_{2n-s },y) \ge r/2 \right) +  2^{1+\beta} C_2 n / r^\beta \\
& \le  C_3n/r^\beta
\end{align*}
for all $x \in X$, $k \in \N^*$ and for all $r> 0$. This immediately implies the desired bound \eqref{e-sur}.
\end{proof}
For $D \subset X$ and a $\mu$-symmetric Markov operator $K$ with kernel $k$, recall from Definition \ref{d-mark} that $K_D$ denotes the sub-Markov operator corresponding to the walk killed upon exiting $D$. As before, we define the exit time of $D$ as  \[ \tau_D = \tau_D^Y = \min \Sett{ n \in \N}{ Y_n \notin D},\]
where $(Y_n)_{n\in\N}$ is the Markov chain corresponding to the operator $K$.
For $D \subset X$, $\lambda >0$, we define corresponding the `resolvent operator' as
\begin{equation} \label{e-resol}
G_\lambda^D f(x) = \left( I - \frac{K_D}{1 + \lambda} \right)^{-1} f(x) = \sum_{i=0}^\infty (1+ \lambda)^{-i}  K_D^i f(x) = \EE_x \sum_{i=0}^{\tau_D} (1+\lambda)^{-i} f(Y_i).
\end{equation}
\begin{lem} \label{l-funh}
Let $(X,d,\mu)$ be a vertex weighted graph satisfying \ref{df} with volume growth exponent $\df$. Let $K$ be a Markov operator satisfying \hyperlink{uhkp}{$\operatorname{UHKP}(\df,\beta)$}. Let $x_0 \in X$, $ r > 10$, $ R >0$ and define the annuli $D_0 = B(x_0,R+ 9r/10) \setminus  B(x_0, R+ r/10), D_1 = B(x_0, R+ 4r/5) \setminus B(x_0,R+ r/5)$, $D_2 = B(x_0 , R+ 3r/5) \setminus B(x_0, R+2r/5)$. Let $\lambda = r^{-\beta}$ and set
\begin{equation}
\label{e-funh1} h= G_\lambda^{D_0} \ind_{D_1},
\end{equation}
where $G_\lambda^{D_0}$ is as defined in \eqref{e-resol}.
Then $h$ is supported in $D_0$ and satisfies
\begin{align}
h(x) & \le 2 r^{\beta} \hspace{3mm} \mbox{ for all $x \in X$,}\label{e-funh2} \\
h(x) & \ge c_1 r^\beta \mbox{ for all $x \in D_2$,}
\end{align}
\end{lem}
\begin{proof}
Since $K_{D_0}$ is a contraction in $L^\infty$, we have
\[
h(x) = \sum_{i=0}^\infty (1+\lambda)^{-i} K_{D_0}^i \ind_{D_1}(x) \le \sum_{i=0}^\infty (1+\lambda)^{-i} = (1+\lambda) \lambda^{-1} \le 2 r^\beta
\]
for all $r \ge 1$, for all $x_0, x \in X$.

Let $(Y_n)_{n\in\N}$ denote the Markov chain driven by $K$.
Let $\epsilon, \delta \in (0,1)$ be given by Lemma \ref{l-sur}. Let $r_0=r/5$, $x \in D_2$ and $B_1=B(x,r_0) \subset D_1$. Then there exists $c_1>0$ such that
for all $x_0 \in X$, for all $r > 10$, for all $x \in D_2$, we have
\begin{align*}
h(x) &= \sum_{i=0}^\infty (1+\lambda)^{-i} K_{D_0}^i \ind_{D_1} (x) \ge \sum_{i=0}^\infty (1+\lambda)^{-i} K_{B_1}^i \ind_{B_1} (x) \\
& \ge (1 + r^{-\beta})^{- \delta r^\beta  } \EE_x \left[ \sum_{i = 0}^ {\lfloor \delta r^\beta \rfloor \wedge \tau^Y_{B_1}}  \ind_{B_1} (Y_i)  \right] \\
& \ge (1 + r^{-\beta})^{- \delta r^\beta  } \delta r^\beta \PP_x( \tau_{B_1}^Y > \delta r^\beta) \ge (1 + r^{-\beta})^{- \delta r^\beta  } \delta r^\beta (1 - \epsilon) \ge c_1 r^\beta.
\end{align*}
\end{proof}

The Dirichlet form corresponding to a $\mu$-symmetric sub-Markov operator $P$ is defined as
\[
\E_P(f,g) := \langle f , (I-P) g \rangle
\]
for all $f,g \in L^2(X,\mu)$, where $\langle \cdot, \cdot \rangle$ denotes the inner product in $L^2(X,\mu)$.
\begin{lem}[Folklore] \label{l-folk} Let $P$ be a $\mu$-symmetric sub-Markov operator with kernel $p$ and let $\E_P$ and $\Gamma$ denote the corresponding Dirichlet form and energy measure respectively.
\begin{enumerate}[(a)]
\item
We have
\begin{equation}\label{e-intp}
\E_P(f,g) = \sum_{x \in X} \Gamma(f,g)(x) + \sum_{x \in X} f(x) g(x) (1 - P \ind (x)) \mu(x)
\end{equation}
for all $f,g \in L^2(X,\mu)$.
In particular, if $P$ is a Markov operator we have
\begin{equation}\label{e-intp1}
\E_P(f,g) =  \langle f , (I-P) g \rangle = \sum_{x \in X} \Gamma(f,g)(x).
\end{equation}
\item The energy measure $\Gamma$ satisfies the integrated version of Leibnitz rule
\begin{equation} \label{e-prod}
\sum_{x \in X} \Gamma(fg,h)(x) = \sum_{x \in X} \left[ f(x) \Gamma(g,h)(x) + g(x) \Gamma(f,h)(x) \right]
\end{equation}
for all bounded functions $f,g,h$, as long as the above sums are absolutely convergent.
\end{enumerate}
\end{lem}
\begin{proof}
Note that
\begin{align} \label{e-flk1}
\nonumber \E_P(f,g) &= \langle f , (I-P) g \rangle = \sum_{x \in X} f(x) g(x) \mu(x) - \sum_{x \in X} f(x) Pg(x) \mu(x) \\
& = \sum_{x \in X} f(x) g(x) \mu(x) ( 1 - P \one(x) + P \one(x)) - \sum_{x \in X} f(x) Pg(x) \mu(x).
\end{align}
By symmetry $p(x,y) = p(y,x)$, we have
\begin{align} \label{e-flk2}
\nonumber \sum_{x \in X} \Gamma(f,g)(x) &= \frac{1}{2} \sum_{x,y} (f(x)-f(y))(g(x) - g(y)) p(x,y) \mu(x) \mu(y) \\
\nonumber &= \frac{1}{2}  \sum_{x \in X} f(x)g(x) P \one(x) \mu(x) + \frac{1}{2}  \sum_{y \in X} f(y)g(y) P \one(y) \mu(y) \\
\nonumber & \hspace{3mm} - \frac{1}{2} \sum_{x \in X} f(x) Pg(x) \mu(x)  - \frac{1}{2} \sum_{y \in X} f(y) Pg(y) \mu(y) \\
&= \sum_{x \in X}  f(x)g(x) P \one(x) \mu(x)- \sum_{x \in X} f(x) Pg(x) \mu(x).
\end{align}
Combining \eqref{e-flk1} and \eqref{e-flk2}, we obtain \eqref{e-intp}.

To prove (b), we follow \cite[Theorem 3.7]{CKS} and write
\begin{align*}
\lefteqn{(f(x)g(x)- f(y)g(y))(h(x)-h(y))}\\
&= \frac{1}{2}(g(x)+g(y))(f(x)-f(y))(h(x)-h(y))  \\
&  \hspace{6mm} + \frac{1}{2}(f(x)+f(y))(g(x)-g(y))(h(x)-h(y)).
\end{align*}
Then an application of symmetry $p(x,y)=p(y,x)$ similar to (a) yields the desired result \eqref{e-prod}.
\end{proof}
The following technical lemma is a consequence of Leibnitz rule above.
\begin{lemma}(\cite[Lemma 3.5]{CKS})\label{l-tech}
Let $(X,d,\mu)$ be a vertex weighted graph and let $P$ with a $\mu$-symmetric sub-Markov operator with kernel $p$ with respect to $\mu$. Let $\Gamma$ denote the corresponding energy measure. If $f,h \in L^2(X,\mu)$ and $g \in L^\infty(X,\mu)$, we have
\begin{equation*}
\sum_{x \in X} g(x) \Gamma(f,h)(x) = \frac{1}{2}  \sum_{x \in X}\left[ \Gamma(gh,f)(x) + \Gamma(gf,h)(x) - \Gamma(g,fh)(x) \right].
\end{equation*}
\end{lemma}
\begin{proof}
We use Leibnitz rule (Lemma \ref{l-folk}(b)) to all terms in the right hand side to obtain the desired equality. The absolute convergence of the various sums are a consequence of H\"{o}lder inequality.
\end{proof}
\begin{remark}
Using the observation $\Gamma(f_1,f_2)(x) = \Gamma(f_1+c_1, f_2+c_2)(x)$, for any $c_1,c_2 \in \mathbb{R}$ and for all $f_1,f_2 \in L^\infty$, we could slightly generalize Lemma \ref{l-tech}.
\end{remark}

\begin{proof}[Proof of Proposition \ref{p-csa}]
If $r \le 10$, we may assume $\phi = \ind_{B(x,R+r/2)}$.
Let $r > 10$ and let $h,c_1$ be as defined in Lemma \ref{l-funh}.
Set
\begin{align}
g(y) &= \frac{h(y)}{c_1 r^\beta} &\mbox{ for all $y \in X$} \label{e-cfs1} \\
\phi(y)&= \begin{cases}
1 \wedge g(y)& \mbox{ if $y\in B(x,R+r/2)^\complement$,} \\
1 & \mbox{ if $y\in B(x,R+r/2)$.}\label{e-cfs2}
\end{cases}
\end{align}
By Lemma \ref{l-funh}, we have \eqref{e-csa1} and \eqref{e-csa2}. It remains to verify \eqref{e-csa3}.
By  Leibnitz rule (Lemma \ref{l-folk}(b), \eqref{e-prod}) and the fact that $\Gamma_U(g,g)$ is supported in $U$, we obtain
\begin{align}\label{e-cfs3}
\sum_{y \in U} f^2(y) \Gamma_U (g,g)(y)  &= \sum_{y \in X} f^2(y)\Gamma_U (g,g)(y)\nonumber \\
&= \sum_{y \in X} \Gamma_U (f^2 g,g)(y)  -  \sum_{y \in X} g(y)\Gamma_U (f^2,g)(y).
\end{align}
As in the Lemma \ref{l-funh}, we set $\lambda= r^{-\beta}$ and $D_1 = B(x_0, R+ 4r/5) \setminus B(x_0,R+ r/5)$. Since $g$ is supported in $D_0= B(x,R+9r/10) \setminus B(x,R+r/10)$, by \eqref{e-intp} of Lemma \ref{l-folk}(a) and Lemma \ref{l-funh}, we obtain
\begin{align}\label{e-cfs4}
\nonumber \sum_{y \in X} \Gamma_U(f^2g,g) & \le \E_{K_U}(f^2 g, g) \le \langle f^2 g, (I - K_U) g \rangle_{L^2(U,\mu)} + \lambda \langle f^2 g, g \rangle_{L^2(U,\mu)} \\
&= \langle f^2 g, ( (1+\lambda)I - K_U) g \rangle_{L^2(U,\mu)} \nonumber \\
& =  \langle f^2 g, ( (1+\lambda)I - K_{D_0}) g \rangle_{L^2(D_0,\mu)} \nonumber\\
 &= (1+\lambda) (c_1 r^\beta)^{-1} \langle f^2 g, ( I - (1+\lambda)^{-1} K_{D_0}) G_\lambda^{D_0} \ind_{D_1} \rangle_{L^2(D_0,\mu)} \nonumber \\
 &= (1+\lambda) (c_1 r^\beta)^{-1} \langle f^2 g, \ind_{D_1} \rangle_{L^2(D_0,\mu)} \nonumber \\
 & \le 4 c_1^{-2} r^{-\beta}\sum_{y \in D_1} f^2(y) \mu(y).
\end{align}
We use Cauchy-Schwarz inequality, the $\mu$-symmetry of $K_U$ and $ab \le a^2/4+ b^2$ to obtain
\begin{align}
 \lefteqn{\abs{\sum_{y \in X} g(y)\Gamma_U (f^2,g)(y)} } \nonumber\\
\nonumber  & =  \frac{1}{2} \abs{\sum_{y,z \in U} g(y) (f^2(y)- f^2(z))(g(y)-g(z)) k(y,z) \mu(y)\mu(z)}\\
\nonumber  & \le \frac{1}{2} \abs{ \sum_{y,z \in U} g(y) f(y) (f(y)- f(z))(g(y)-g(z)) k(y,z) \mu(y)\mu(z)}\\
&\hspace{4mm} + \frac{1}{2} \abs{ \sum_{y,z \in U} g(y) f(z) (f(y)- f(z))(g(y)-g(z)) k(y,z) \mu(y)\mu(z)} \nonumber\\
& \le \frac{1}{4} \sum_{y \in U}f^2(y) \Gamma_U(g,g)(y)  + \sum_{y \in U} g^2(y) \Gamma_U(f,f)(y) \nonumber\\
 &\hspace{4mm} + \frac{1}{4} \sum_{z \in U} f^2(z) \Gamma_U(g,g)(z) + \sum_{y \in U} g^2(y) \Gamma_U(f,f)(y) \nonumber\\
 & = \frac{1}{2} \sum_{y \in U} f^2(y) \Gamma_U(g,g)(y) + 2 \sum_{y \in U} g^2(y) \Gamma_U(f,f)(y). \label{e-cfs5}
\end{align}
Combining the above, we obtain
\begin{align*}
  \sum_{y \in U}f^2(y) \Gamma(\phi,\phi)(y) & \le \sum_{y \in U}f^2(y) \Gamma(g,g)(y) \\
 &\le 2 \sum_{y \in X} \Gamma_U (f^2g,g)(y) + 4 \sum_{y \in U} g^2(y) \Gamma_U(f,f)(y) \\
& \le 4 c_1^{-2} \sum_{y \in U} \Gamma_U(f,f)(y) + 8 c_1^{-2} r^{-\beta} \sum_{y \in U} f^2(y) \mu(y).
\end{align*}
For the first line above we use $\abs{\phi(y) - \phi(z)} \le \abs{g(y) - g(z)}$ , the second line follows from  \eqref{e-cfs3} and \eqref{e-cfs5}, and the last line
follows from   \eqref{e-cfs4} and   $g \le c_1^{-1}$
\end{proof}
The ``linear cutoff function" for $B(x,r) \subset B(x,2r)$
\begin{equation} \label{e-linear}
\psi(y) = \max\left(\min\left(1, \frac{2r -  d(x,y)}{r}\right), 0 \right)
\end{equation}
is commonly used to obtain off-diagonal estimates using Davies' method (see, for example, \cite{Dav2,BGK}). We will first see how this linear cutoff functions compares to the ones obtained in Proposition \ref{p-csa} satisfying inequality \eqref{e-csa3}.
\begin{lem} \label{l-linear}
Let $(X,d,\mu)$ be a vertex weighted graph satisfying \ref{df}. Let $K$ be a $\mu$-symmetric Markov operator satisfying \hyperlink{j-beta}{$\operatorname{J}(\beta)$}, for some $\beta \ge 2$. Let  $\psi$ is the cutoff function in \eqref{e-linear} for $B(x,R) \subset B(x,R+r)$ for some $x \in X$ and $r \ge 1$.  Then there exists $C_1>0$ such that the corresponding energy measure $\Gamma(\psi,\psi)$  satisfies the inequality
\begin{equation} \label{e-lcf}
  \Gamma(\psi,\psi)(y) \le \begin{cases} \frac{C_1}{r^2} & \mbox{if $\beta >2$} \\
\frac{C_1 \log(1+r)}{r^2} & \mbox{if $\beta =2$} \end{cases}
\end{equation}
for all $y \in X$,
where $C_1$ does not depend on $x \in X$, $r\ge 1$ or $R > 0$.
\end{lem}
\begin{proof}
Let $k$ denote the $\mu$-symmetric kernel of $K$.
Note that $\psi$ is $1/r$-Lipschitz function and $\abs{\psi(y)-\psi(z)} \le 1$ for all $y,z \in X$. Therefore
\begin{align} \label{e-lcf1}
\nonumber \Gamma(\psi,\psi)(y) & =  \sum_{z \in X}  (\psi(y)-\psi(z))^2 k(y,z) \mu(z)\mu(y) \\
& \le   \sum_{z \in B(y,r)} r^{-2} d(y,z)^2 k(y,z) \mu(z)\mu(y) + \sum_{z \notin B(y,r)} k(y,z) \mu(z)\mu(y).
\end{align}
To bound  the second term above, by \ref{j-beta} and \ref{df} there exists $C_2,C_3 >0$ such that
\begin{align} \label{e-lcf2}
\sum_{z \notin B(y,r)} k(y,z) \mu(z)\mu(y)
&  \le   \sum_{z \notin B(y,r)} \frac{C_2 \mu(z) }{(d(y,z))^{\df+\beta}} \nonumber \\
& =   \sum_{i=1}^\infty \sum_{z:  2^{i-1}r <d(y,z) \le 2^{i}r}\frac{C_2 \mu(z) }{(d(y,z))^{\df+\beta}} \nonumber
\le  \sum_{i=1}^\infty\frac {C_2 V(y,2^{i}r)}{ (2^{i-1}r)^{\df+\beta}} \\
& \le C_3 r^{-\beta} \le C_3 r^{-2}
\end{align}
for all $y \in X$ and for all $r\ge 1$.

For the first term in \eqref{e-lcf1}, by \ref{j-beta} and \ref{df}  there exists $C_4,C_5>0$ such that
\begin{align} \label{e-lcf3}
\nonumber \sum_{z \in B(y,r)}  d(y,z)^2 k(y,z) \mu(z)\mu(y) &\le \sum_{i=0}^{\lceil \log_2 r \rceil}   \sum_{z \in X: 2^i \le d(y,z)< 2^{i+1}} d(y,z)^2 k(y,z) \mu(z)\mu(y)\\
\nonumber & \le C_4  \sum_{i=0}^{\lfloor \log_2 r \rfloor}  2^{i(2 - \beta)}  \\
&\le \begin{cases}  C_5 \log (1+r) & \mbox{if $\beta=2$} \\
C_5 & \mbox{if $\beta>2$}
\end{cases}
\end{align}
for all $y \in X$ and for all $r\ge 1$.
Combining \eqref{e-lcf1}, \eqref{e-lcf2} and \eqref{e-lcf3}, we obtain the desired bound \eqref{e-lcf}.
\end{proof}
%self-improving property of CSA and various truncated versions
For a $\mu$-symmetric Markov operator $K$ with kernel $k(\cdot,\cdot)$ and for bounded function $f$, we define the energy measure corresponding to a truncation at scale $L>0$ as
\begin{equation}\label{e-trunc}
\Gamma_{(L)}(f,f)(x) = \frac{1}{2} \sum_{y \in B(x,L)} (f(x)-f(y))^2  k(x,y) \mu(y) \mu(x).
\end{equation}
Next, we show a self-improving property of the cutoff-Sobolev inequality that will play an important role in the next section.
\begin{prop} \label{p-csj}
	Let $(X,d,\mu)$ be a vertex weighted graph satisfying \ref{df} with volume growth exponent $\df$.
	Let $K$ be $\mu$-symmetric Markov operator whose kernel with respect to $\mu$ satisfies \ref{j-beta} for some $\beta \ge 2$. If $K$ satisfies \hyperlink{csj}{$\operatorname{CSJ}(\beta)$}, then $K$ satisfies the following estimate:
	there exists $C_1,C_2>0$
	such that for all $n \in \N^*$,  for all $x \in X$, for all $r \ge 1$, there exists a cut-off function $\phi_n$ for $B(x,r) \subset B(x,2r)$ such that
	\begin{equation} \label{e-csj1}
	\sum_{y \in X} f^2(y) \Gamma(\phi_n,\phi_n)(y) \le \frac{C_1}{n^2} \sum_{y \in X} \Gamma(f,f)(y)  + \frac{C_2 \G_\beta(n)}{r^\beta} \sum_{y \in X} f^2(y) \mu(y),
	\end{equation}
	where the function $\G_\beta$ is given by
	\begin{equation}\label{e-defG}
	\G_\beta(n) = \begin{cases} n^{\beta-2} & \mbox{if $\beta >2$} \\
	\log(1+n) &\mbox{ if $\beta=2$}.
	\end{cases}
	\end{equation}
	Further the function $\phi_n$ above satisfies
	\begin{equation}\label{e-csj2}
	\norm{\phi_n-\psi}_\infty \le n^{-1},
	\end{equation}
	where $\psi$ is the linear cutoff function given by
	\[
	\psi(y) = \max\left(\min\left(1, \frac{2r -  d(x,y)}{r}\right), 0 \right).
	\]
\end{prop}
\begin{proof}
Let $\Gamma$ denote the energy measure of $K$.
By Lemma \ref{l-linear}, without loss of generality we may assume $r > 10 n$.
We divide the annulus $U= B(x,2r) \setminus B(x,r)$ into $n$-annuli $U_1,U_2,\ldots,U_n$ of equal width, where
\[
U_i:= B(x,r+  ir/n) \setminus  B(x,r+(i-1)r/n), \hspace{5mm} i=1,2,\ldots,n.
\]
By Proposition \ref{p-csa},  there exists $C_3,C_4>0$ and cutoff functions $\phi_{(i)}$ for $B(x,r+(i-1)r/n) \subset B(x,r+ir/n)$ satisfying $0 \le \phi_{(i)} \le 1$,
\begin{align}\label{e-sip1}
\phi_{(i)} &\equiv 1 \hspace{6mm}\mbox{in $B(x,r+(i-1)r/n+r/(2n))$,}\\
\phi_{(i)} &\equiv 0 \hspace{6mm}\mbox{in $B(x,r+(i-1)r/n + 9r/(10n))^\complement$,}\label{e-sip2} \\
\sum_{y \in U_i} f^2(y) \Gamma_{U_i}(\phi_{(i)},\phi_{(i)})(y) &\le C_3 \sum_{y \in U_i} \Gamma_{U_i}(f,f)(y) + \frac{C_4}{ (r/n)^\beta} \sum_{U_i} f^2(y) \mu(y)\label{e-sip3}
\end{align}
for $i=1,2,\ldots,n$. We define $\phi_n= n^{-1} \sum_{i=1}^n \phi_{(i)}$, $s= r/(10n)$ and $\Gamma_{(s)}$ truncated energy measure at scale $s$ as given by \eqref{e-trunc}.
Note that $\phi_n$ satisfies \eqref{e-csj2} because $\phi_n(y) \in [1-i/n, 1-(i-1)/n]$ for all $y \in U_i$ and for all $i=1,2,\ldots,n$.
 By \eqref{e-sip1} and \eqref{e-sip2}, we have
\begin{equation} \label{e-sip4}
\Gamma_{(s)}(\phi_n,\phi_n)(y) = n^{-2} \sum_{i=1}^n \Gamma_{(s)}(\phi_{(i)},\phi_{(i)})(y) \le n^{-2}\sum_{i=1}^n \Gamma_{U_i}(\phi_{(i)},\phi_{(i)})(y)
\end{equation}
for all $y \in X$.
Combining \eqref{e-sip3} and \eqref{e-sip4}, we obtain
\begin{equation}\label{e-sip5}
\sum_{y \in X} f^2(y) \Gamma_{(s)}(\phi_n,\phi_n)(y) \le C_3 n^{-2} \sum_{y \in X} \Gamma(f,f)(y) + \frac{C_4 n^{\beta-2}}{r^\beta} \sum_{y \in X} f^2(y)\mu(y).
\end{equation}
To bound $\Gamma(\phi_n,\phi_n)(y) - \Gamma_{(s)}(\phi_n,\phi_n)(y)$, we write
\begin{equation} \label{e-sip6}
\Gamma(\phi_n,\phi_n) - \Gamma_{(s)}(\phi_n,\phi_n) =  \Gamma(\phi_n,\phi_n) - \Gamma_{(r)}(\phi_n,\phi_n) + \Gamma_{(r)}(\phi_n,\phi_n)- \Gamma_{(s)}(\phi_n,\phi_n).
\end{equation}
 Since $0 \le \phi_n \le 1$ by \eqref{e-lcf2} there exists $C_5>0$ such that
\begin{equation}\label{e-sip7}
\Gamma(\phi_n,\phi_n)(y)  - \Gamma_{(r)}(\phi_n,\phi_n)(y) \le \sum_{y \in B(x,r)^\complement} k(x,y) \mu(x)\mu(y) \le C_5 r^{-\beta} 
\end{equation}
for all $r>0$ and $n \in \N^*$.
Note that for all $y,z \in X$ such that $d(y,z) \ge s =r/(10n)$, by \eqref{e-csj2} we have $\abs{\phi_n(y)-\phi_n(z)} \le r^{-1} d(y,z) +2/n \le  21 r^{-1} d(y,z)$.
Therefore by a similar calculation as \eqref{e-lcf3}, there exists $C_6>0$ such that, for all $y \in X$, for all $n \in \N^*$ and for all $r > 10n$, we have
\begin{align}\label{e-sip8}
 \lefteqn{\Gamma_{(r)}(\phi_n,\phi_n)(y)- \Gamma_{(s)}(\phi_n,\phi_n)(y) }\nonumber \\
 &= \frac{1}{2} \sum_{z \in B(y,r) \setminus B(y,s)} (\phi_n(y)- \phi_n(z))^2 k(y,z) \mu(z)\mu(y) \nonumber \\
&\le  \frac{21^2}{2} \sum_{z \in B(y,r) \setminus B(y,s)} r^{-2} d(y,z)^2 k(y,z) \mu(z)\mu(y) \nonumber \\
& \le   \frac{21^2}{2} \sum_{i =0}^{\lfloor \log_2(10n) \rfloor} \sum_{z \in B(y,2^{i+1}s) \setminus B(y,2^is)} r^{-2} d(y,z)^2 k(y,z) \mu(z)\mu(y)  \nonumber \\
&\le C_6 \frac{ \G_\beta(n)}{ r^\beta}.
\end{align}
Combining \eqref{e-sip5}, \eqref{e-sip6}, \eqref{e-sip7} and \eqref{e-sip8}, we obtain
\begin{equation}\label{e-sip9}
\sum_{y \in X} f^2(y) \Gamma(\phi_n,\phi_n)(y) \le \frac{C_3}{n^2} \E_K(f,f)+ (C_4+C_6+C_7) \frac{\G_\beta(n)}{r^\beta} \sum_{y \in X} f^2(y) \mu(y)
\end{equation}
for all $x \in X$, for all $n \in \N^*$ and for all $r > 10n$. Combining \eqref{e-sip6} with Lemma \ref{l-linear} yields the desired result.
\end{proof}
\section{Davies' method}\label{s-dav}
In this section, we carry out the Davies perturbation method to obtain heat kernel upper bounds \hyperlink{uhkp}{$\operatorname{UHKP}(\df,\beta)$} for heavy tailed walks satisfying \ref{j-beta} and the
cutoff Sobolev inequality  \hypertarget{csj}{$\operatorname{CSJ}(\beta)$}, for some $\beta \ge 2$.
For the case the $\beta \in (0,2)$, in \cite{BGK} the heat kernel upper bounds for the corresponding continuous time process was obtained. The corresponding discrete time bounds were obtained in \cite{MS1}.

The idea behind the approach of \cite{BGK} is to use Meyer's construction \cite{Mey} to split the jump kernel intro small and large sums and then apply Davies' method for the smaller jumps (see \cite[Section 3]{BGK}).
However as mentioned in the introduction (see \cite[Section 1]{GHL}),
 the above approach is no longer adequate to obtain \hyperlink{uhkp}{$\operatorname{UHKP}(\df,\beta)$} for the case $\beta \ge 2$.
The goal of this section is to modify Davies' perturbation method to obtain upper bounds for heavy tailed jump processes satisfying \ref{j-beta} for the case $\beta \ge 2$.
The following Proposition is the converse of Proposition \ref{p-csj}.
\begin{prop}\label{p-offd}
Let $(X,d,\mu)$ be a vertex weighted graph satisfying  \ref{df} with volume growth exponent $\df$.
Let $K$ be $\mu$-symmetric Markov operator whose kernel with respect to $\mu$ satisfies \ref{j-beta} for some $\beta \ge 2$. If $K$ satisfies the cutoff Sobolev inequality \hyperlink{csj}{$\operatorname{CSJ}(\beta)$}, then $K$ satisfies the transition probability upper bounds \hyperlink{uhkp}{$\operatorname{UHKP}(\df,\beta)$}.
\end{prop}
We introduce two definitions below.
\begin{definition}
We say that a $\mu$-symmetric sub-Markov operator $T$ on a  graph $(X,d)$ is $L$-local if its corresponding kernel $t$ satisfies $t(x,y)=0$ whenever $x,y \in X$ satisfies $d(x,y) > L$.
\end{definition}
\begin{definition}
Let $\psi$ be a function on a graph $(X,d)$ and let $L >0$. We define the oscillation of $\psi$ at scale $L$ as
\[
\osc(\psi,L) := \sup_{x,y \in X: d(x,y) \le L} \abs{\psi(x)-\psi(y)}.
\]
\end{definition}
The following Lemma is an analogue of \cite[Theorem 3.9]{CKS}. The computations are similar to the ones in \cite{CKS} but we will use a different strategy to control the energy measure at various places.
\begin{lem}\label{l-leib}
Let $T$ be a $\mu$-symmetric, $L$-local, sub-Markov operator on a vertex weighted graph $(X,d,\mu)$ and let $\Gamma$ denote the corresponding energy measure. Then for any function $\psi \in L^\infty(X,\mu)$  with bounded support, for all $p \ge 1$ and for all $f \in L^2(X,\mu)$ with $f \ge 0$, we have
\begin{align} \label{e-dv}
\MoveEqLeft{\sum_{x \in X} \Gamma(e^{-\psi}f,e^\psi f^{2p-1})(x)} \nonumber \\
  &\ge \frac{1}{2p} \sum_{x \in X} \Gamma(f^p,f^p)(x) - 9p e^{2 \osc(\psi,L)} \sum_{x \in X} f^{2p}(x)\Gamma(\psi,\psi)(x).
\end{align}
\end{lem}
\begin{proof}
Let $t$ denote the kernel of $T$ with respect to $\mu$.
Then by Lemma \ref{l-tech}, we have
\begin{align} \label{e-dv1}
\sum_{x \in X} \Gamma (e^{- \psi} f,  e^{\psi} f^{2p-1})(x) & = \sum_{x \in X} \left( \Gamma(e^{-\psi} f^{2p}, e^{\psi})(x) +  \Gamma(f,f^{2p-1})(x) \right) \nonumber \\
& \hspace{6mm} - 2 \sum_{x \in X} e^{- \psi(x)} f(x) \Gamma(e^{\psi}, f^{2p-1})(x).
\end{align}
A diligent reader will observe that one cannot directly apply Lemma \ref{l-tech}
because $e^\psi \notin L^2(X,\mu)$. However, since $\psi$ has bounded support, $e^\psi -1 \in L^2(X,\mu)$ and we can apply Lemma \ref{l-tech} with $f,g,h$ replaced by $e^\psi-1,e^{-\psi}f,f^{2p-1}$ respectively. Then we use the remark following Lemma \ref{l-tech} to obtain \eqref{e-dv1}. 

We use Leibniz rule (Lemma \ref{l-folk}(b)) for the first term in the right in \eqref{e-dv1} to obtain
\begin{align} \label{e-dv2}
\sum_{x \in X} \Gamma(e^{- \psi} f,  e^{ \psi} f^{2p-1})(x)
 & = \sum_{x \in X}f^{2p-1}(x) \Gamma(e^{- \psi} f, e^{ \psi})(x) \\
\nonumber & \hspace{5mm} +  \sum_{x \in X} \left( \Gamma(f,f^{2p-1})(x)  - e^{-\psi(x)} f(x) \Gamma (e^{\psi}, f^{2p-1})(x) \right).
\end{align}
Note that
\begin{align} \label{e-dv3}
\MoveEqLeft{\sum_{x \in X}
\left(  f^{2p-1}(x) \Gamma(e^{- \psi} f, e^{ \psi})(x)  -e^{-\psi(x)} f(x) \Gamma (e^{\psi}, f^{2p-1})(x) \right)}  \\
& = \frac{1}{2} \sum_{x,y \in X} \left\{ \left(e^{-\psi(x)}f(x) f^{2p-1}(y) - e^{-\psi(y)} f(y) f^{2p-1}(x) \right) \right. \nonumber \\
& \hspace{40mm} \times \left. \left(e^{\psi(x)} - e^{\psi(y)} \right) t(x,y) \mu(x) \mu(y) \right\}. \nonumber
\end{align}
We define $a(x,y):= t(x,y) \mu(x) \mu(y)$ for all $x,y \in X$.
Next, we rewrite the right side of \eqref{e-dv3} as
\begin{align}\label{e-dv4}
\MoveEqLeft{\frac{1}{2} \sum_{x,y \in X} (f^{2p}(y)- f^{2p}(x))e^{-\psi(x)}(e^{\psi(x)}-e^{\psi(y)}) a(x,y)} \nonumber \\
&+ \frac{1}{2}  \sum_{x,y \in X} f^{2p}(x) (e^{-\psi(x)}-e^{-\psi(y)}) (e^{\psi(x)}-e^{\psi(y)}) a(x,y) \nonumber \\
&+ \sum_{x,y \in X} f^{2p-1}(y) (f(x)-f(y)) e^{\psi(y)} (e^{-\psi(y)} - e^{-\psi(x)})a(x,y).
\end{align}
For the first term in \eqref{e-dv4},  we use Cauchy-Schwartz inequality to obtain
\begin{align} \label{e-dv5}
\MoveEqLeft{\frac{1}{2} \sum_{x,y \in X} (f^{2p}(y)- f^{2p}(x))e^{-\psi(x)}(e^{\psi(x)}-e^{\psi(y)}) a(x,y)} \nonumber \\
&= \frac{1}{2} \sum_{x,y \in X} f^p(y) (f^{p}(y)- f^{p}(x))e^{\psi(y)}(e^{-\psi(y)}-e^{-\psi(x)}) a(x,y) \nonumber \\
& + \frac{1}{2} \sum_{x,y \in X} f^p(x) (f^{p}(y)- f^{p}(x))e^{-\psi(x)}(e^{\psi(x)}-e^{\psi(y)}) a(x,y) \nonumber \\
&\ge - \left(\sum_{x \in X} \Gamma(f^p,f^p)(x) \right)^{1/2} \left[ \left( \sum_{y \in X} f^{2p}(y) e^{2\psi(y)} \Gamma(e^{-\psi},e^{-\psi})(y) \right)^{1/2} \right. \nonumber \\
& \hspace{40mm}+\left. \left(\sum_{x \in X} f^{2p}(x) e^{-2\psi(x)} \Gamma(e^{\psi},e^{\psi})(x) \right)^{1/2} \right].
\end{align}
For the second term in \eqref{e-dv4},  we use Cauchy-Schwartz inequality to obtain
\begin{align} \nonumber
\MoveEqLeft{\frac{1}{2}  \sum_{x,y \in X} f^{2p}(x) (e^{-\psi(x)}-e^{-\psi(y)}) (e^{\psi(x)}-e^{\psi(y)}) a(x,y)= \sum_{x \in X} f^{2p}(x) \Gamma(e^{-\psi},e^\psi)(x)}\\
\MoveEqLeft{ \ge - \left(\sum_{x \in X} f^{2p}(x) e^{2\psi(x)} \Gamma(e^{-\psi},e^{-\psi})(x) \right)^{1/2} \left(\sum_{x \in X} f^{2p}(x) e^{-2\psi(x)} \Gamma(e^{\psi},e^{\psi})(x) \right)^{1/2}}.\label{e-dv6}
\end{align}
For the last term in \eqref{e-dv4}, we use Cauchy-Schwartz inequality to obtain
\begin{align} \label{e-dv7}
\MoveEqLeft{\sum_{x,y \in X} f^{2p-1}(y) (f(x)-f(y)) e^{\psi(y)} (e^{-\psi(y)} - e^{-\psi(x)})a(x,y)} \\
 &\ge - 2\left( \sum_{x \in X} f^{2p-2}(x) \Gamma(f,f)(x) \right)^{1/2} \left( \sum_{x \in X} f^{2p}(x) e^{2 \psi(x)} \Gamma(e^{-\psi},e^{-\psi})(x)\right)^{1/2}. \nonumber
\end{align}
We need two more elementary inequalities (Cf. (3.16) and (3.17) in \cite{CKS}). For any non-negative function $f \in L^2(X,\mu)$ and for all $p \ge 1$, we have
\begin{equation}\label{e-dv8}
\sum_{x \in X} \Gamma(f^{2p-1},f)(x) \ge \sum_{x \in X} f^{2p-2}(x) \Gamma(f,f)(x) \ge \frac{1}{2p-1} \sum_{x\in X}  \Gamma(f^{2p-1},f)(x)
\end{equation}
and
\begin{equation}\label{e-dv9}
\sum_{x \in X} \Gamma(f^p,f^p)(x) \ge \sum_{x \in X}\Gamma(f^{2p-1},f)(x)  \ge \frac{2p-1}{p^2} \sum_{x \in X} \Gamma(f^p,f^p)(x).
\end{equation}
Since $T$ is $L$-local, we use the inequality  $(e^a-1)^2 \le e^{2 \abs{a}} a^2$ to obtain
\begin{equation}\label{e-dv10}
\max\left(e^{2\psi(x)} \Gamma(e^{-\psi},e^{-\psi})(x) ,e^{-2\psi(x)}  \Gamma(e^{\psi},e^{\psi})(x) \right)\le e^{2 \osc(\psi,L)} \Gamma(\psi,\psi)(x)
\end{equation}
for all $x \in X$.
Combining  equations \eqref{e-dv2} through \eqref{e-dv10}, we have
\begin{align}
\MoveEqLeft{\sum_{x \in X} \Gamma(e^{-\psi}f,e^\psi f^{2p-1})(x)} \nonumber \\
  &\ge \frac{2p-1}{p^2} \sum_{x \in X} \Gamma(f^p,f^p)(x) - e^{2 \osc(\psi,L)} \sum_{x \in X} f^{2p}(x)\Gamma(\psi,\psi)(x) \nonumber \\
&\hspace{5mm} - 4 e^{ \osc(\psi,L)} \left(\sum_{x \in X} \Gamma(f^p,f^p)(x) \right)^{1/2}   \left( \sum_{x \in X} f^{2p}(x)\Gamma(\psi,\psi)(x) \right)^{1/2}. \nonumber
&\hspace{5mm}
\end{align}
We use the inequality $4ab \le a^2/(2p) + 8p b^2$  and $p \ge 1$ to obtain
\begin{align*}
\MoveEqLeft{\sum_{x \in X} \Gamma(e^{-\psi}f,e^\psi f^{2p-1})(x)} \nonumber \\
  &\ge \frac{1}{2p} \sum_{x \in X} \Gamma(f^p,f^p)(x) - 9p e^{2 \osc(\psi,L)} \sum_{x \in X} f^{2p}(x)\Gamma(\psi,\psi)(x) \nonumber \\
\end{align*}
for all bounded functions $\psi$ and for all non-negative functions $f \in L^2(X,\mu)$.
\end{proof}
\begin{remark}
The elementary inequalities \eqref{e-dv8} and \eqref{e-dv9} are sometimes called Stroock-Varopolous inequalities. An unified approach to such inequalities is provided in \cite[Lemma 2.4]{MOS13}.
\end{remark}
The next step is to bound the second term in the right side of \eqref{e-dv}
using cutoff Sobolev inequalities developed in Section \ref{s-csj}.
\begin{lem} \label{l-cutfun}
Let $(X,d,\mu)$ be a vertex weighted graph satisfying \ref{df} with volume growth exponent $\df$.
Let $K$ be a $\mu$-symmetric Markov operator whose kernel $k=k_1$ satisfies \ref{j-beta} for some $\beta  \ge 2$.   Further assume that $K$ satisfies the cutoff Sobolev inequality \hyperlink{csj}{$\operatorname{CSJ}(\beta)$}. Let $\Gamma$ denote the energy measure corresponding to $K$ and let $\Gamma_{(L)}$ correspond to the truncated version of $\Gamma$ for some $L>0$.
 Define
\[
\vartheta:= \frac{1}{8(\df+\beta)}.
\]
Then there exists $\lambda_0  \ge 1$, $C_0 \ge 1$ such that the following property holds: For all $\lambda \ge \lambda_0$, for all $p \ge 1$, for all $x \in X$, for all $r \ge 1$, there exists a cut-off function $\phi= \phi_{p,\lambda}$  for $B(x,r) \subset B(x,2r)$ such that for all non-negative functions $f \in L^2(X,\mu)$ we have
\begin{equation} \label{e-cfn1}
 \sum_{y \in X} \Gamma_{(\vartheta r)}(e^{-\lambda \phi} f, e^{\lambda \phi} f^{2p-1}) \ge \frac{1}{4p} \sum_{y \in X} \Gamma(f^p,f^p)(y)  - \frac{C_0  p^\beta e^{8 \beta \vartheta \lambda}}{r^\beta} \sum_{y \in X} f^{2p}(y) \mu(y),
\end{equation}
where the cutoff function $\phi$ above satisfies
\begin{equation}\label{e-cfn2}
\norm{\phi-\psi}_\infty \le \frac{1}{\lambda p},
\end{equation}
where $\psi$ is given by \eqref{e-linear}.
\end{lem}
\begin{proof}
By Lemma \ref{l-leib}, for any cutoff function $\phi$ for $B(x,r)\subset B(x,2r)$, for all $p \ge 1$, for all $\lambda >0$,  and for all non-negative $f \in L^2(X,\mu)$ we have
\begin{align}\label{e-cn1}
\MoveEqLeft{  \sum_{y \in X} \Gamma_{(\vartheta r)}(e^{-\lambda \phi} f, e^{\lambda \phi} f^{2p-1})} \nonumber \\
& \ge  \frac{1}{2p} \sum_{y \in X} \Gamma_{(\vartheta r)}(f^p,f^p)(y) - 9p \lambda^2 e^{2 \lambda \osc(\phi,\vartheta r)} \sum_{y \in X} f^{2p}(y)\Gamma(\phi,\phi)(y).
\end{align}
Using  \ref{j-beta} and a similar computation as \cite[Equation (33)]{MS1}, there exists $C_2>0$ such that
\begin{equation} \label{e-cn2}
\sum_{y \in X} \left( \Gamma(f^p,f^p)(y)-\Gamma_{(\vartheta r)}(f^p,f^p)(y) \right)\le C_2 r^{-\beta} \sum_{y \in X} f^{2p}(y)\mu(y)
\end{equation}
for all $r \ge 1$, for all $p \ge 1$ and for all $f \in L^2(X,\mu)$.
By \hyperlink{csj}{$\operatorname{CSJ}(\beta)$}, there exists $C_3,C_4>1$ such that for any $n \in \N^*$ there exists a cutoff function $\phi= \phi_n$ for $B(x,r) \subset B(x,2r)$ such that for all $f \in L^2(X,\mu)$ with $f \ge 0$, we have
\begin{equation} \label{e-cn3}
 \sum_{y \in X} f^2(y) \Gamma(\phi,\phi)(y) \le \frac{C_3}{n^2} \sum_{y \in X} \Gamma(f,f)(y)  + \frac{C_4 \G_\beta(n)}{r^\beta} \sum_{y \in X} f^2(y) \mu(y),
\end{equation}
where the function $\phi=\phi_n$ above satisfies
\begin{equation}\label{e-cn4}
\norm{\phi-\psi}_\infty \le n^{-1}.
\end{equation}
We make the choice
\begin{equation} \label{e-nc}
n= \left\lceil 6 p \lambda \exp(3 \lambda \vartheta) \sqrt{C_3} \right\rceil
\end{equation}
 and $\lambda_0 \ge 1$ such that $n \ge \vartheta^{-1}$ for all $\lambda \ge \lambda_0$ and for all $p \ge 1$.
We will verify $\phi$ satisfies the desired properties \eqref{e-cfn1} and \eqref{e-cfn2}.
 Using \eqref{e-nc} and \eqref{e-cn4}, we immediately have \eqref{e-cfn2}.
By \eqref{e-cn4}, triangle inequality, $n \ge \vartheta^{-1}$ and by the fact that $\psi$ is $r^{-1}$-Lipschitz,  we have
\begin{equation}
\label{e-cn5}
\osc(\phi,\vartheta r) \le \osc(\psi,\vartheta r) + 2 n^{-1} \le  \vartheta  + 2 \vartheta = 3 \vartheta.
\end{equation}

By \eqref{e-nc}, \eqref{e-cn3}, \eqref{e-cn5}, we have
\begin{align} \label{e-cn6}
\MoveEqLeft{9 p \lambda^2 e^{2 \lambda \osc(\phi,\vartheta r)} \sum_{y \in X} f^{2p}(y) \Gamma(\phi,\phi)(y)}\nonumber \\
& \le \frac{1}{4p} \sum_{y \in X} \Gamma(f^p,f^p)(y)  + \frac{9 C_4 p \lambda^2 e^{6 \vartheta \lambda} \G_\beta(n)}{r^\beta} \sum_{y \in X} f^{2p}(y) \mu(y),
\end{align}
for all non-negative $f \in L^2(X,\mu)$. We use $n \le 12 p \lambda e^{3 \lambda \vartheta} \sqrt{C_3}$, $\lambda \le \vartheta^{-1} e^{\vartheta \lambda}$ and  $\G_\beta(n) \le n^{\beta -1}$ for all $n \in \N^*$ to obtain $C_5 >0$ (that depends on $\vartheta,\beta$) such that
\begin{equation} \label{e-cn7}
9 C_4 p \lambda^2 e^{6 \vartheta \lambda} \G_\beta(n) \le 9 C_4 p \lambda^2 e^{6 \vartheta \lambda} (12 p \lambda e^{3 \lambda \vartheta} \sqrt{C_3})^{\beta-1} \le C_5 p^\beta e^{ 4(\beta+1) \lambda \vartheta} \le  C_5 p^\beta e^{ 8\beta \lambda \vartheta}
\end{equation}
for all $p,n \ge 1$ and for all $\lambda > \lambda_0$.
Combining \eqref{e-cn1}, \eqref{e-cn2}, \eqref{e-cn6} and \eqref{e-cn7}, we obtain \eqref{e-cfn1}.
\end{proof}
We need the following Nash inequality to obtain off-diagonal estimates using Davies' method.
\begin{prop}[Nash inequality] \label{p-nash}
Let $(X,d,\mu)$ be a vertex weighted graph satisfying satisfies \ref{df} with volume growth exponent $\df$.  Let $K$ be a $\mu$-symmetric Markov operator whose kernel $k$ satisfies \ref{j-beta} for some $\beta > 0$. Let $\E(\cdot,\cdot)$ denote the corresponding Dirichlet form. Then there exists $C_N>0$ such that
\begin{equation} \label{e-nash}
\norm{f}_2^{2(1+ \beta/\df)} \le C_N \E(f,f) \norm{f}_1^{2\beta/\df}
\end{equation}
for all $f \in L^1(X,\mu)$, where $\norm{\cdot}_{p}$ denotes the $L^p(X,\mu)$ norm.
\end{prop}
\begin{proof}
The proof of the Nash inequality \eqref{e-nash} is essentially contained in \cite[p.~1064]{BCLS}.
We repeat the proof for convenience.

For $r > 0$ and $f \in L^1(X,\mu)$, we denote by $f_r: X \to \R$ the $\mu$-weighted average
\[
f_r(x) = \frac{1}{V(x,r)} \sum_{y \in B(x,r)} f(y) \mu(y).
\]
We first bound $\norm{f_r}_1$. There exists $C_V>0$ such that
\begin{align}\label{e-ns1}
\norm{f_r}_1 &= \sum_{x \in X} \abs{f_r(x)} \mu(x) \nonumber \\
& \le \sum_{x \in X} \sum_{y \in B(x,r)} \abs{f(y)} V(x,r)^{-1} \mu(y) \mu(x) = \sum_{y \in X} \abs{f(y)} \mu(y) \sum_{x \in B(y,r)}V(x,r)^{-1}  \mu(x) \nonumber\\
& \le C_V^2 \norm{f}_1
\end{align}
for all $f \in L^1(X,\mu)$ and for all $r > 0$. The second line above follows from triangle inequality and Fubini's theorem and the last line follows from \ref{df}.
By \ref{df}, there exists $C_1>0$ such that
\begin{equation}\label{e-ns2}
\norm{f_r}_\infty \le \norm{f}_1 \sup_{x \in X} V(x,r)^{-1} \le C_1 r^{-\df} \norm{f}_1
\end{equation}
$f \in L^1(X,\mu)$ and for all $r > 0$.
Combining \eqref{e-ns1} and \eqref{e-ns2}, there exists $C_2>0$ such that
\begin{equation}\label{e-ns3}
\norm{f_r}_2 \le \norm{f_r}_1^{1/2} \norm{f_r}_\infty^{1/2}  \le C_2 r^{-\df/2} \norm{f}_1
\end{equation}
for all $f \in L^1(X,\mu)$ and for all $r > 0$.

There exists $C_3 >0$ such that  for all $f \in L^1(X,\mu)$ and for all $r > 0$, we have
\begin{align}\label{e-ns4}
\norm{f-f_r}_2^2 &= \sum_{x \in X} \abs{f(x)- f_r(x)}^2 \mu(x)\nonumber\\
%& = \sum_{x \in X}\left( \frac{1}{V(x,r)}  \sum_{y \in X} (f(x)-f(y)) \mu(y)  \right)^2 \mu(x) \nonumber\\
& \le \sum_{x \in X}\frac{1}{V(x,r)} \sum_{y \in B(x,r)}  (f(x)-f(y))^2   \mu(y) \mu(x)\nonumber \\
& \le C_1 \sum_{x,y \in X} (f(x)-f(y))^2 \one_{\set{d(x,y) \le r}} r^{-\df} \mu(x) \mu(y)\nonumber\\
& \le 2^{\df + \beta} C_1 r^\beta \sum_{x,y \in X} (f(x)-f(y))^2 \frac{1}{(1+d(x,y))^{\df+\beta}} \mu(x) \mu(y) \nonumber \\
& \le \frac{C_3 r^\beta}{2} \sum_{x,y \in X} (f(x)-f(y))^2 k(x,y) \mu(x) \mu(y) = C_3 r^\beta \E(f,f).
\end{align}
The second line above follows from Jensen's inequality, the third line follows from an application of  \ref{df} similar to \eqref{e-ns2}
 and the last line follows from \ref{j-beta} and Lemma \ref{l-folk}(a).

 By triangle inequality, \eqref{e-ns3} and \eqref{e-ns4}, we have
 \begin{equation}\label{e-ns5}
 \norm{f}_2 \le \norm{f_r}_2 + \norm{f-f_r}_2 \le C_2 r^{-\df/2} \norm{f}_1 + C_3^{1/2} r^{\beta/2} \E(f,f)^{1/2}
 \end{equation}
 for all $f \in L^1(X,\mu)$ and all $r>0$. The Nash inequality \eqref{e-nash} follows from \eqref{e-ns5} and the choice
 \[
 r = \left(\frac{\norm{f}_1^2}{\E(f,f)}\right)^{1/(\df+\beta)}.
 \]
\end{proof}

The following lemma is analogous to \cite[Lemma 3.21]{CKS} but the statement and its proof is slightly modified to suit our context.
\begin{lemma}\label{l-dife}
Let $w:[0,\infty) \to (0,\infty)$ be a  non-decreasing function and suppose that $u \in C^1([0,\infty); (0,\infty))$ satisfies
\begin{equation} \label{e-cks}
 u'(t) \le - \frac{\epsilon}{p} \left( \frac{t^{(p-2)/\theta p}}{w(t)}\right)^{\theta p } u^{1+\theta p}(t) + \delta p^{\beta}u(t)
\end{equation}
for some positive $\epsilon, \theta$ and $\delta$, $\beta \in [2,\infty)$ and $p=2^k$ for some $k \in \N^*$. Then  $u$ satisfies
\begin{equation} \label{e-fs} %fs= Fabes-Stroock
 u(t) \le  \left( \frac{2 p^{\beta+1} }{\epsilon \theta} \right)^{1/\theta p} t^{(1-p)/\theta p} w(t)  e^{\delta t/p}.
\end{equation}
\end{lemma}
\begin{proof}
 Set $v(t) = e^{- \delta p^{\beta}t} u(t)$. By \eqref{e-cks}, we have
 \[
  v'(t)= e^{- \delta p^{\beta}t} \left( u'(t) - \delta p^{\beta} u(t) \right) \le - \frac{\epsilon t^{p-2}}{p w(t)^{\theta p}} e^{\theta \delta p^{\beta+1}t} v(t)^{1+\theta p}.
 \]
Hence
\[
 \frac{d}{dt} \left( v(t) \right)^{-\theta p} \ge \epsilon \theta t^{p-2} w(t)^{-\theta p} e^{\theta \delta p^{\beta+1}t}
\]
and so, since $w$ is non-decreasing
\begin{equation} \label{e-fs1}
 e^{ \delta \theta p^{\beta+1} t} u(t)^{-\theta p} \ge \epsilon \theta w(t)^{-\theta p} \int_0^t s^{(p-2)}  e^{\theta \delta p^{\beta+1}s} \, ds.
\end{equation}
Note that
\begin{eqnarray}
\nonumber \int_0^t s^{(p-2)}  e^{\theta \delta p^{\beta+1}s} \, ds &\ge& (t/\delta \theta p^{\beta+1})^{p-1} \int_{\delta \theta p^{\beta+1}(1- 1/p^{\beta+1})}^{\delta \theta p^{\beta+1}} y^{(p-2)} e^{ty}\, dy \\
\nonumber &\ge & \frac{t^{p-1}}{p-1} \exp \left( \delta \theta p^{\beta+1} t -  \delta \theta t \right) \left[ 1 - (1- p^{-\beta-1})^{p-1} \right] \\
 \label{e-fs2} &\ge & \frac{t^{p-1}}{2p^{\beta+1}} \exp \left( \delta \theta p^{\beta+1} t - \delta \theta  t \right).
\end{eqnarray}
In the last line above, we used the bound $1- (1-p^{-\beta-1})^{p-1} \ge \frac{1}{2}p^{-\beta-1}(p-1)$ for all $p,\beta \ge 2$.
Combining \eqref{e-fs1} and \eqref{e-fs2} yields \eqref{e-fs}.
\end{proof}
We now have all the necessary ingredients to prove Proposition \ref{p-offd}.
\begin{proof}[Proof of Proposition \ref{p-offd}]
Let $H_t:= \exp( t(K-I))$ denote the corresponding continuous time semigroup. Note that $H_t$ is $\mu$-symmetric and let $h_t$ be the kernel of $H_t$ with respect to $\mu$.
As explained in \cite[Remark 3]{MS1}, by \cite[Theorem 3.6]{Del} it suffices to show the following bound on $h_t$: there exists $C_1>0$ such that
\begin{equation} \label{e-od1}
h_t(x,y) \le C_1 \min \left( \frac{1}{t^{\df/\beta}}, \frac{t}{(d(x,y))^{\df+\beta}} \right)
\end{equation}
for all $x,y \in X$ and for all $t \ge 1$. By \cite[Proposition II.1]{Cou} the Nash inequality \eqref{e-nash} in Proposition \ref{p-nash} implies that there exists $C_2>0$ such that
\begin{equation} \label{e-od2}
h_t(x,y) \le \frac{C_2 }{t^{\df/\beta}}
\end{equation}
for all $x,y \in X$ and for all $t \ge 1$.

Let $t \ge 1$ and let $x, y \in X$ such that $d(x,y) \ge 2$. Let $\vartheta= 1/(8(\df+\beta))$ and $\lambda_0,C_0 \ge 1$ be constants given by Lemma \ref{l-cutfun}.
We define parameters $L$ and $r$ that depend only on $x$ and $y$ as
\begin{equation} \label{e-defod}
r=d(x,y)/2, L = \vartheta r.
\end{equation}
 Let $\lambda \ge \lambda_0$ be arbitrary.
 Let $\Gamma$ denote the energy measure corresponding to $K$ and let $\Gamma_{(L)}$ denote its truncation.
Define the truncated Dirichlet form
\[
\E_L(f,f)= \sum_{y \in X} \Gamma_{(L)}(f,f)(y)
\]
for all $f \in L^2(X,\mu)$. Let $h^{(L)}_t$ denote the  continuous time kernel with respect to $\mu$ for the corresponding jump process and let $H^{(L)}_t$ denote the corresponding Markov semigroup. Let $p_k=2^k$ for $k \in \N$ and let $\psi_k= \lambda \phi_{p_k,\lambda}$ where $\phi_{p_k,\lambda}$ is the cutoff function for $B(x,r) \subset B(x,2r)$ given by  Lemma \ref{l-cutfun}.
Define the `perturbed semigroup'
\begin{equation} \label{e-od3}
H_t^{L,\psi_k}f(x) = e^{\psi_k(x)} \left(H_t^{(L)} \left( e^{-\psi_k} f \right)\right)(x).
\end{equation}

We now pick $f \in L^2(X,\mu)$ and $f \ge 0$ with $\norm{f}_2=1$ and define
\begin{equation} \label{e-fdef}
 f_{t,k}:= H_t^{L,\psi_{k} } f
\end{equation}
for all $k \in \N$, and $P_t^{\psi_{k} }$ denotes the perturbed semigroup defined in \eqref{e-od3}.  We remark that the constants below do not depend on the choice of $x,y \in X$ with $d(x,y) \ge 2$  or the choice of $f \in L^2(X,\mu)$,  $t \ge 1$ or $\lambda \ge \lambda_0$.

Using Lemma \ref{l-cutfun}, we have
\begin{align}
\nonumber \frac{d}{dt} \norm{f_{t,0}}_{2}^2 =& -2 \E_L\left(e^{\psi_0} f_{t,0}, e^{- \psi_0 } f_{t,0}\right) \\
\le& 2 C_0 \frac{ e^{8 \beta \vartheta \lambda} }{r^\beta} \norm{f_{t,0}}_2^2  \label{e-sr1}
\end{align}
and
\begin{align}
\nonumber \frac{d}{dt} \norm{f_{t,k}}_{2p_{k}}^{2p_k} &= -2 p_k \E_L\left(e^{\psi_k}  f_{t,k}^{2p_k-1}, e^{-\psi_k }  f_{t,k} \right) \\
&\le - \frac{1}{2}\E \left( f_{t,k}^{p_k}, f_{t,k}^{p_k}\right) + 2 C_0 \frac{e^{8 \beta \vartheta \lambda} p_k^{\beta+1}}{r^\beta} \norm{f_{t,k}}_{2 p_k}^{2 p_k} \label{e-sr2}
\end{align}
for all $k \in \N^*$.
By solving \eqref{e-sr1}, we obtain
\begin{equation}
 \label{e-sr3} \norm{f_{t,0}}_{p_1} =\norm{f_{t,0}}_2 \le \exp\left(  C_0  e^{8 \beta \vartheta \lambda} t /r^\beta \right) \norm{f}_2= \exp\left(  C_0  e^{8 \beta \vartheta \lambda} t /r^\beta \right).
\end{equation}
Using \eqref{e-sr2} and Nash inequality \eqref{e-nash}, we obtain
\begin{equation}
\label{e-sr4} \frac{d}{dt} \norm{f_{t,k}}_{2p_{k}} \le - \frac{1}{4 C_N p_k} \norm{f_{t,k}}_{2 p_k}^{1+ 2 \beta p_k /\df }\norm{f_{t,k}}_{p_k}^{-2 \beta p_k/\df} + C_0 p_k^{\beta} \frac{ e^{8 \beta \vartheta \lambda}}{r^\beta}\norm{f_{t,k}}_{2p_k}
\end{equation}
for all $k \in \N^*$. By \eqref{e-cfn2}, we have $\norm{\psi_k -\psi_{k-1}}_\infty \le 3/ p_k$. This along the fact that $H_t^{(L)}$ is a contraction on $L^\infty$ yields
\begin{equation} \label{e-sc}
  \exp(-6/p_k) f_{t,k-1}  \le f_{t,k} \le \exp(6/p_k) f_{t,k-1}
\end{equation}
for all $k  \in \N^*$.
Combining \eqref{e-sr4} and \eqref{e-sc}, we obtain
\begin{equation}
 \label{e-sr5}  \frac{d}{dt} \norm{f_{t,k}}_{2p_{k}} \le - \frac{1}{C_A p_k} \norm{f_{t,k}}_{2 p_k}^{1+ 2 \beta p_k /\df }\norm{f_{t,k-1}}_{p_k}^{-2 \beta p_k/\df} + C_0 p_k^{\beta} \frac{e^{8 \beta \vartheta \lambda}}{r^\beta}\norm{f_{t,k}}_{2p_k}
\end{equation}
for all $k \in \N^*$, where $C_A= 4 C_N \exp ( 12 \beta /\df)$.

Let $u_{k}(t)= \norm{f_{t,k-1}}_{p_{k}}$  and  let  \[w_k(t)= \sup \{ s^{ \df(p_k-2)/(2 \beta p_k)} u_k(s) : s \in (0,t] \}.\]
By \eqref{e-sr3}, $w_1(t) \le \exp(C_0e^{8 \beta \vartheta \lambda}  t /r^\beta)$. Further by \eqref{e-sr5}, $u_{k+1}$ satisfies \eqref{e-cks} with
$\epsilon = 1/C_A$, $\theta = 2 \beta/\df$, $\delta = C_0 (e^{8 \beta \vartheta \lambda}/r^\beta)$, $w=w_k$, $p=p_k$. Hence by Lemma \ref{l-dife},
\[
 u_{k+1}(t) \le ( 2^{(\beta+1) k +1}/\epsilon \theta)^{1/(\theta p_k)} t^{(1-p_k)/\theta p_k} e^{\delta t/p_k} w_k(t).
\]
Therefore
\[
 w_{k+1}(t)/w_k(t) \le ( 2^{(\beta+1) k +1}/\epsilon \theta)^{1/(\theta 2^k)} e^{\delta t/2^k}
\]
for $k \in \N^*$. Hence, we obtain
\[
 \lim_{k \to \infty} w_{k}(t) \le   C_3 e^{\delta t}   w_1(t) \le C_2 \exp(2 C_0 e^{8 \beta \vartheta \lambda} t/r^\beta )
\]
where $C_3 =C_3( \beta, \epsilon, \theta)> 0$.
Hence
\[
 \lim_{k \to \infty} u_k(t) = \norm{ H_t^{L,\psi_\infty} f}_\infty \le  \frac{C_2}{t^{\df/2\beta}} \exp(2 C_0 e^{8 \beta \vartheta \lambda} t/r^\beta ).
\]
where $\psi_\infty=\lim_{k \to \infty} \psi_k$.
Since the above bound holds for all non-negative $f \in L^2(X,\mu)$, we obtain
\[
 \norm{H_t^{L,\psi_\infty}  }_{2 \to \infty} \le  \frac{C_3}{t^{\df/2\beta}} \exp(2 C_0 e^{8 \beta \vartheta \lambda} t/r^\beta ).
\]
The estimate is unchanged if we replace $\psi_k$'s by $-\psi_k$. Since $H_t^{L,-\psi}$ is the adjoint of $H_t^{L,\psi}$, by duality we have that for all $t >0$
\[
 \norm{H_t^{L,\psi_\infty}}_{1\to 2} \le  \frac{C_3}{t^{\df/2\beta}} \exp(2 C_0 e^{8 \beta \vartheta \lambda} t/r^\beta ).
\]
Combining the above, we have
\begin{equation*}
 \norm{H_t^{L,\psi_\infty} }_{1\to \infty} \le  \frac{C_3 2^{\df/\beta}}{t^{\df/\beta}} \exp(2 C_0e^{8 \beta \vartheta \lambda} t/r^\beta ).
\end{equation*}
 for all $\lambda \ge \lambda_0$, for all $x,y \in M$ with $d(x,y) \ge 2$,  where $L$ is as defined in \eqref{e-defod}.
Therefore
\begin{equation} \label{e-sr6}
 h_t^{(L)}(x,y) \le  \frac{C_2 2^{\df/\beta}}{t^{\df/\beta}} \exp(2 C_0 e^{8 \beta \vartheta \lambda} t/r^\beta + \psi_\infty(y) - \psi_\infty(x) )
\end{equation}
for all $x,y \in X$ satisfying $d(x,y) \ge 2$,  for all $t \ge 1$ and for all $\lambda \ge \lambda_0$.

Now we choose $\lambda$ as
\begin{equation} \label{e-defpar}
 \lambda = \frac{\df+\beta}{\beta} \log (r^\beta/t).
\end{equation}
 By \eqref{e-od2}, it suffices to show \eqref{e-od1} for the case $d(x,y) \ge 2$ and $\lambda \ge \lambda_0$.
By the choices in \eqref{e-defod}, we have $\psi_\infty(y) - \psi_\infty(x)= - \lambda$ and hence by  \eqref{e-sr6}, we obtain
\begin{equation}\label{e-sr7}
 h_t^{(L)}(x,y) \le  \frac{C_2 2^{\df/\beta}}{t^{\df/\beta}} \exp(2 C_0 e^{8 \beta \vartheta \lambda} t/r^\beta - \lambda ) =  \frac{C_2 2^{\df/\beta}}{t^{\df/\beta}} \exp(2 C_0) \frac{t^{1+(\df/\beta)}}{r^{\df+\beta}}.
\end{equation}
By \eqref{e-sr7}, \eqref{e-defod} along with Meyer decomposition bound in \cite[Lemma 3.1]{BGK}, there exists $C_4,C_5>0$ such that for all $x,y \in X$ and for all $t \ge 1$ satisfying $d(x,y)^\beta \ge C_5 t$, we have
\[
h_t(x,y) \le \frac{C_4 t}{ (d(x,y))^{\df+\beta}}.
\]
The above equation along with \eqref{e-od2} yields \eqref{e-od1} which in turn implies \hyperlink{uhkp}{$\operatorname{UHKP}(\df,\beta)$}.
\end{proof}
We will now prove the main results stated in the introduction.
\begin{proof}[Proof of Theorem \ref{t-main3}]
Proposition \ref{p-offd} and Proposition \ref{p-csj} are the two desired implications.
\end{proof}
\begin{proof}[Proof of Corollary \ref{c-main2}]
Let $K_1$ satisfy \ref{hkp}. Then by Theorem \ref{t-main3},  we have that $K_1$ satisfies the cutoff Sobolev inequality \hyperlink{csj}{$\operatorname{CSJ}(\beta)$}.
Since two $\mu$-symmetric Markov operators satisfying \ref{j-beta} have comparable energy measures,  $K_2$ satisfies \hyperlink{csj}{$\operatorname{CSJ}(\beta)$}.
Therefore by Theorem \ref{t-main3},  $K_2$ satisfies \hyperlink{uhkp}{$\operatorname{UHKP}(\df,\beta)$}.
The upper bounds \hyperlink{uhkp}{$\operatorname{UHKP}(\df,\beta)$}, \eqref{e-od1}, and the lower bound of the kernel in \ref{j-beta} are sufficient to show the matching lower bounds \hyperlink{lhkp}{$\operatorname{LHKP}(\df,\beta)$} using an iteration argument due to Bass and Levin \cite{BL}. The argument in \cite[Sections 4 and 5]{MS1} can be directly adapted to this setting. Therefore $K_2$ satisfies \ref{hkp}.
\end{proof}
\subsection{Applications} \label{ss-example}

The following transition probability estimate is the main application of our stability result and provides several examples.
\begin{theorem}\label{t-main}
Let $(X,d,\mu)$ be a vertex weighted graph satisfying \ref{df} with volume growth exponent $\df$. Let $(S_n)_{n \in \mathbb{N}}$ denote the simple random on $X$ with transition probability $P_n(x,y)= \mathbb{P}\left(X_n=y \hspace{1mm}\vline\hspace{1mm} X_0=x\right)$. Suppose that the transition probability $P_n$ satisfies the following sub-Gaussian estimates \ref{usg} and \ref{lsg} with walk dimension $\dw$: 
there exists constants $c,C>0$ such that, for all $x,y \in X$
\begin{equation}\label{usg}
 P_n(x,y) \le \frac{C}{n^{\df/\dw}} \exp \left[ - \left( \frac{d(x,y)^\dw}{Cn} \right)^{\frac{1}{\dw-1}} \right], \forall n \ge 1; \tag*{$\operatorname{USG}(\df,\dw)$}
\end{equation}
and
\begin{equation}\label{lsg}
(P_n+P_{n+1})(x,y) \ge \frac{c}{n^{\df/\dw}}  \exp \left[ - \left( \frac{d(x,y)^\dw}{cn} \right)^{\frac{1}{\dw-1}} \right], \forall n \ge 1 \vee d(x,y). \tag*{$\operatorname{LSG}(\df,\dw)$}
\end{equation} Let $K$ be a $\mu$-symmetric Markov operator whose kernel $k=k_1$ satisfies \ref{j-beta} for some $\beta \in [2,\dw)$.
Then the corresponding iterated kernel $k_n$ satisfies \ref{hkp}.
\end{theorem}
\begin{remark}
\begin{enumerate}[(a)]
\item For the case $\beta \in (0,2)$, \ref{hkp} follows from \ref{df} and \ref{j-beta} (see \cite[Theorem 1.1]{MS1}).
In this case there is no assumption required on the heat kernel of $(X,d,\mu)$.
\item For the case $\beta \ge \dw$, \ref{hkp} does not hold.
We refer the reader to \cite[Theorem 1.2]{MS2} for on-diagonal estimates.
 We do not know off-diagonal estimates in this case.
\item We obtained matching two sided estimates on $\sup_{x \in X} k_{2n}(x,x)$ in   \cite{MS2} for all $\beta > 0$. Due to Theorem \ref{t-main}, we have analogous pointwise on-diagonal lower bounds on $k_n(x,x)$ for all $x \in X$ for the case $\beta \in [2,\dw)$.
\item The hypothesis of sub-Gaussian estimate for the simple random walk can be generalized to any random walk satisfying certain uniform ellipticity hypothesis -- see  \cite[Theorem 1.2]{MS2} for such a set up.
\item The application of Davies' method in literature (see for example \cite[Section 3.2]{BGK}) is inadequate to obtain \hyperlink{uhkp}{$\operatorname{UHKP}(\df,\beta)$} for the case $\beta \ge 2$. If $\beta >2$, the existing methods yields the off-diagonal upper bound corresponding to the off-diagonal estimate in  \hyperlink{uhkp}{$\operatorname{UHKP}(\df,2)$} and therefore not optimal by a factor of $(1+d(x,y))^{\beta-2}$. Even in the case $\beta = 2$, the existing method gives an off-diagonal upper bound that is not optimal by the logarithmic factor $\log(2+d(x,y))$.
\end{enumerate}
\end{remark}
\begin{proof}[Proof of Theorem \ref{t-main}]
By a known  subordination argument  (see \cite[Theorem 5.1]{MS2}), there exists a $\mu$-symmetric Markov operator satisfying \ref{j-beta} and \ref{hkp}.
The desired result then follows from Corollary \ref{c-main2}.
\end{proof}
In the remainder of this section, we will elaborate on a rich family of examples using Theorem \ref{t-main}.
\begin{example} \label{x-one}
For any $\df \in [1,\infty)$ and for any $\dw \in [2,\df+1]$,
Barlow constructs graphs that satisfy sub-Gaussian estimates \ref{usg} and \ref{lsg} -- see  \cite[Theorem 2]{Bar} and  \cite[Theorem 3.1]{GT}.
Moreover, these are the complete range of $\df$ and $\dw$ for which sub-Gaussian estimates \ref{usg} and \ref{lsg} could possibly hold for graphs.

Let us fix one such graph $(X,d)$ satisfying \ref{usg} and \ref{lsg} for some $\df \in [1,\infty),\dw \in (2,\df+1]$.
Consider any function $J:X \times X \to (0,\infty)$ satisfying 
\begin{align} \label{e-jb}
\frac{C^{-1}}{(1+d(x,y))^{\df+\beta}} \le J(x,y)&= J(y,x) \le \frac{C}{(1+d(x,y))^{\df+\beta}}\hspace{3mm} \forall x,y \in X. 
\end{align}
Any function $J$ satisfying \eqref{e-jb} defines a measure $\mu\left(\set{x} \right) = \sum_{y \in X} J(x,y)$ and a $\mu$-symmetric Markov operator $K$ with kernel 
$k(x,y)= J(x,y) \left( \mu(x) \mu(y) \right)^{-1}$. If $\beta \in [2,d_w)$, we obtain \ref{hkp} for the Markov operator $K$. 

To compare with some earlier works in the continuous time case, let us point out that we obtain continuous time heat kernel estimates corresponding to \hyperlink{uhkp}{$\operatorname{UHKP}(\df,\beta)$} in \eqref{e-od1}. Consider a continuous time jump process with symmetric measure $\nu$ satisfying \ref{df}  and the Dirichlet form 
\begin{equation}
\mathcal{E}(f,f) = \int_X \int_X \left(f(x)-f(y) \right)^2 J(x,y) \nu(dx) \nu(dy)
\end{equation}
on $L^2(X,\nu)$, where $J$ satisfies \eqref{e-jb}. The condition \eqref{e-jb} on the jump kernel $J$ of the Dirichlet form should be interpreted as the continuous time analog of \ref{j-beta}. By comparison of energy measures  and the measures $\mu$ and $\nu$, we obtain Nash and cut-off Sobolev inequalities for the Dirichlet form $\mathcal{E}$ with symmetric measure $\nu$. By the same argument as the proof of \eqref{e-od1}, we obtain the continuous time analogue of \hyperlink{uhkp}{$\operatorname{UHKP}(\df,\beta)$} for the heat kernel of the continuous time process associated with the above Dirichlet form $\mathcal{E}$ on $L^2(X,\nu)$. 

Now, let us consider a graph $(X,d)$ satisfying \ref{usg} and \ref{lsg}  with $d_f=100,d_w=101$. %  -- in this case, we may even assume that $(X,d)$ is a Vicsek-type tree as constructed in \cite[Propostion 4]{Bar}. 
As mentioned above, the heat kernel corresponding to $\left(\mathcal{E}, L^2(X,\nu) \right)$ satisfies the continuous time analogue of \hyperlink{uhkp}{$\operatorname{UHKP}(\df,\beta)$} for all $\beta \in [2,101)$. We note that the results of \cite{GHL} imply the continuous time analogue of \hyperlink{uhkp}{$\operatorname{UHKP}(\df,\beta)$} for the case $\beta \in (100,101)$ -- see \cite[Corollary 6.14]{GHL}. However, in the case $\beta \in [2,100)$ the exit time and survival time estimates appearing in \cite[Theorems 2.1 and 2.3]{GHL} are \emph{a priori} difficult to verify.
\end{example}

\subsection*{Acknowledgement}
We thank Martin Barlow for inspiring conversations and for his helpful suggestions on an earlier draft. We thank Tom Hutchcroft for proofreading part of the manuscript.

As the authors were working on this project, another group, Zhen-Qing Chen, Takashi Kumagai, and Jian Wang, were pursuing similar goals based on somewhat different ideas \cite{CKW}. Both groups obtain similar final results regarding the heat kernel estimates \ref{hkp}, independently. We believe that both works will prove useful in future progress. The work \cite{CKW} concerns similar estimates in the continuous time setting and allows for more general volume growth and space-time scaling. Further their work involves a weaker version of cutoff Sobolev inequality. A crucial step in their proof involves a mean value inequality proved using an iteration procedure by repeated application of Faber-Krahn and cutoff-Sobolev inequalities.

\emph{Added in revision}: A more recent preprint of Grigor'yan, Hu, and Hu \cite{GHH}  also independently addresses similar questions using a different approach. We would also like to point out that \cite{CKW2,CKW3} addresses the stability of Harnack inequalities and the relationship between Harnack inequalities and heat kernel estimates in the context of jump processes.
\bibliographystyle{amsalpha}

\end{document}